\documentclass[a4paper]{amsart}
\usepackage{amssymb,amsmath}
\title{Twistors, $4$-symmetric spaces and integrable systems}
\author{Francis E. Burstall}
\address{Department of Mathematical Sciences\\ University of Bath\\
Bath BA2 7AY\\UK}
\email{feb@maths.bath.ac.uk}
\author{Idrisse Khemar}
\address{TU Munich, Zentrum Mathematik (M8)\\Boltzmannstr. 3\\D-85747 Garching\\
GERMANY}
\email{khemar@ma.tum.de}
\thanks{Part of this work was carried out during the workshop
" Surface Theory: Research in Pairs" at Kloster Schoental, March 2008.}
\theoremstyle{plain}

\newtheorem{thm}{Theorem}[section]

\newtheorem{lemma}[thm]{Lemma}

\theoremstyle{definition}

\theoremstyle{remark}
\newtheorem*{rem}{Remark}

\newcommand{\g}{\mathfrak{g}}
\renewcommand{\k}{\mathfrak{k}}
\newcommand{\p}{\mathfrak{p}}
\newcommand{\h}{\mathfrak{h}}
\newcommand{\so}{\mathfrak{so}}
\newcommand{\C}{\mathbb{C}}
\newcommand{\R}{\mathbb{R}}
\newcommand{\Oct}{\mathbb{O}}
\DeclareMathOperator{\Hom}{Hom}
\DeclareMathOperator{\End}{End}
\DeclareMathOperator{\Aut}{Aut}
\DeclareMathOperator{\Fix}{Fix}
\DeclareMathOperator{\ad}{ad}
\DeclareMathOperator{\Ad}{Ad}
\DeclareMathOperator{\trace}{trace}
\newcommand{\submfd}{\Sigma}
\renewcommand{\d}{\mathrm{d}}
\newcommand{\II}{\mathrm{I\kern-1ptI}}
\newcommand{\nb}{\overline{\nabla}}
\newcommand{\ns}{\nabla^\submfd}
\newcommand{\np}{\nabla^\perp}
\newcommand{\half}{\tfrac{1}{2}}
\newcommand{\set}[1]{\{#1\}}
\newcommand{\mcv}{\mathbf{H}}

\begin{document}
\maketitle

\section{Introduction}
\label{sec:introduction}

The theory of harmonic maps of surfaces has been greatly enriched by
ideas and methods from integrable systems
\cite{BurFerPed93,Hit90,Poh76,Uhl89,Zakilo78,ZakSha78}.  More
recently, H\'elein--Romon \cite{HelRom02,HelRom02A,HelRom05} have
applied similar ideas in their study of Hamiltonian stationary
Lagrangian surfaces in $4$-dimensional Hermitian symmetric spaces.
The integrable system arising in this latter theory is a special case
of one discussed by Terng \cite{Ter02} for which a geometric
interpretation seemed lacking.  It is the purpose of this paper to
remedy that lack.

The underlying algebraic structure of the situation is a Lie algebra
equipped with an automorphism of order $4$.  Geometrically, this
means we have to do with a (locally) $4$-symmetric space and the
integrable system we study amounts to an equation on maps from a
Riemann surface into this space.  We begin by observing the
$4$-symmetric spaces may be viewed as submanifolds of the twistor
space of an associated Riemannian symmetric space (this is further
elaborated in \cite{Khe08}) and then our equation becomes the demand
that the map into twistor space be a vertically harmonic twistor
lift.  When the Riemannian symmetric space is $4$-dimensional,
twistor lifts of conformal immersions are canonically defined and we
show that vertical harmonicity amounts to holomorphicity of the mean
curvature vector.  In particular, we find that conformal immersions
of Riemann surfaces in $4$-dimensional space-forms with holomorphic
mean curvature vector constitute an integrable system.

\section{An integrable system}
\label{sec:an-integrable-system}
Let us begin by describing an integrable system that has arisen in
the general theory of Terng \cite{Ter02} and work of H\'elein--Romon
\cite{HelRom02,HelRom02A,HelRom05} and Khemar \cite{Khe05}.

Our first ingredient is a Lie algebra $\g$ together with an
automorphism $\tau\in\Aut(\g)$ of order $4$.  We have an
eigenspace decomposition
\begin{equation}
\label{eq:1}
\g^\C=\g_0\oplus\g_2\oplus\g_1\oplus\g_{-1}
\end{equation}
so that $\tau$ has eigenvalue $e^{2\pi ik/4}$ on $\g_k$.

Now let $\submfd$ be a Riemann surface and
$\alpha\in\Omega^1_\submfd\otimes\g$ be a $\g$-valued $1$-form on
$\submfd$.  The complex structure $J^\submfd$ of $\submfd$ gives a
type decomposition $\alpha=\alpha^{1,0}+\alpha^{0,1}$ while
\eqref{eq:1} gives a second decomposition:
\begin{equation*}
\alpha=\alpha_0+\alpha_2+\alpha_1+\alpha_{-1}.
\end{equation*}

With this in hand, our integrable system, the \emph{second elliptic
$(\g,\tau)$-system} of \cite{Ter02}, comprises the following
equations on $\alpha$:
\begin{subequations}
\label{eq:2}
\begin{gather}
\alpha^{0,1}_{1}=0 \label{eq:2a}\\
\d\alpha^{1,0}_2+[\alpha_0\wedge\alpha^{1,0}_2]=0
\label{eq:2b}\\
\d\alpha+\tfrac{1}{2}[\alpha\wedge\alpha]=0.\label{eq:2c}
\end{gather}
\end{subequations}
The equations \eqref{eq:2} have a \emph{zero-curvature}
formulation: for $\lambda\in\C^\times$, define
$\alpha_\lambda\in\Omega^1_\submfd\otimes\g^\C$ by
\begin{equation*}
\alpha_\lambda=\lambda^{2}\alpha^{1,0}_2+\lambda\alpha^{1,0}_{1}+
\alpha_0+\lambda^{-1}\alpha^{0,1}_{-1}+\lambda^{-2}\alpha^{0,1}_2.
\end{equation*}
Then, in the presence of \eqref{eq:2a}, equations
\eqref{eq:2b} and \eqref{eq:2c} are equivalent to the demand
that
$\d\alpha_\lambda+\tfrac{1}{2}[\alpha_\lambda\wedge\alpha_\lambda]=0$
(thus $\d+\alpha_\lambda$ is a flat connection), for all
$\lambda\in\C^\times$.

Thus the methods of integrable system theory (see, for example,
\cite{BraDor06}) apply to give generalised Weierstrass formulae,
algebro-geometric solutions, spectral deformations and so on.

Our purpose in this note is to describe the geometry behind this
integrable system.
  
\section{$4$-symmetric spaces and twistor spaces}
\label{sec:4-symmetric-spaces}
Let $\sigma=\tau^2$ so that $\sigma^2=1$ and we have a symmetric
decomposition
\begin{equation*}
\g=\k\oplus\p
\end{equation*}
with $\k^\C=\g_0\oplus\g_2$ and $\p^\C=\g_{-1}\oplus\g_1$.  Without
loss of generality, we assume that $\ad_{|\p}:\k\to\End(\p)$ injects
(any kernel is a $\tau$-invariant ideal of $\g$ that we factor out).
We then have:
\begin{subequations}
\label{eq:3}
\begin{align}
\label{eq:4}
\g_0&=\set{\xi\in\g^\C:[\ad\xi_{|\p},\tau_{|\p}]=0}\\
\label{eq:5}
\g_2&=\set{\xi\in\g^\C:\lbrace\ad\xi_{|\p},\tau_{|\p}\rbrace=0}
\end{align}
\end{subequations}
where here, and below, $\lbrace\,,\,\rbrace$ is anti-commutator.

Now let us integrate our setup: let $(G,K)$ be a symmetric pair
corresponding to $(\g,\sigma)$, that is, $G$ is a Lie group with Lie
algebra $\g$ with an involution $\sigma\in\Aut(G)$ integrating our
$\sigma\in\Aut(\g)$ and $(G^\sigma)_0\leq K\leq G^\sigma$.  Thus
$G/K$ is a symmetric space and we assume henceforth that $K$ is
compact so that $G/K$ is a Riemannian symmetric space.

Define $H\leq G$ by
\begin{equation*}
H=\set{h\in K\colon \Ad(h)\circ\tau_{|\p}=\tau_{|\p}\circ\Ad(h)}
\end{equation*}
From \eqref{eq:3}, we deduce that $H$ is the stabiliser of $\tau$ in
$K$ and so has Lie algebra $\h=\g_0\cap\g$.  Thus $G/H$ is a locally
$4$-symmetric space\footnote{In fact, in our examples below, $G$ will
admit $\tau\in\Aut(G)$ with $\tau^4=1$ integrating $\tau$ and $H$
will be open in $\Fix(\tau)$ so that $G/H$ is globally
$4$-symmetric.}.  By construction, $G/H$ fibres over the Riemannian
symmetric space $G/K$ and we claim that this fibration factors
through the twistor fibration of $G/K$.

For this, recall that the twistor space $J(N)$ of an even-dimensional
Riemannian manifold is the bundle of orthogonal almost complex
structures on $TN$:
\begin{equation*}
J(N)=\{j\in\so(TN): j^2=-1\}. 
\end{equation*}
In the case at hand, we have the standard identification $T_{eK}G/K=\p$ so
that $\tau_{|\p}$ is such an orthogonal\footnote{That $\tau_{|\p}$ is
orthogonal for some $K$-invariant inner product on $\p$ is the
content of \cite[Theorem~21]{Khe08}.} almost complex structure. 
Extending by equivariance, we see that the fibration $\pi: G/H\to
G/K$ factors $G/H\hookrightarrow J(G/K)\to G/K$. 

\begin{rem}
In \cite{Khe08} it is shown that, up to covers, essentially all
locally $4$-symmetric spaces arise as submanifolds of the twistor
space of a Riemannian symmetric space in this way.
\end{rem}

Given a Riemannian symmetric space $G/K$, it is natural to ask which
$G$-orbits in $J(G/K)$ arise by the above procedure.  There is a
necessary condition: under the identification $T_{eK}G/K=\p$, the
curvature operator at $eK$ is given by $R_{eK}(X,Y)=-\ad[X,Y]_{|\p}$
so that, since $\tau\in\Aut(\g)$, we have, for $j=\tau_{|\p}$,
\begin{equation}
\label{eq:7}
R(jX,jY)=j\circ R(X,Y)\circ j^{-1},
\end{equation}
for $X,Y\in T_{\pi(j)}G/K$.  In fact, in most cases, this condition is also
sufficient:
\begin{thm}[\cite{Khe08}]\label{th:1}
Let $G/K$ be a simply-connected Riemannian symmetric space on which
$G$ acts effectively and $j\in J(G/K)$.  Then the $G$-orbit of $j$ is
a locally $4$-symmetric space arising as above from some
$\tau\in\Aut(\g)$ of order $4$ if and only if \eqref{eq:7} holds.
\end{thm}
\begin{rem}
In fact, Theorem~\ref{th:1} holds in rather more generality.  A
problem arises only for flat spaces of the form $\R^k\times T^{n-k}$
where the curvature condition is vacuous but orbits by the isometry
group of those $j$ which do not respect the product structure fail to
be locally $4$-symmetric.  See \cite[Theorem~10]{Khe08} for more
details.
\end{rem}

\section{Vertically harmonic twistor lifts}
\label{sec:vert-harm-twist}

Let us now return to the integrable system \eqref{eq:2}: a solution
$\alpha$ of \eqref{eq:2} integrates, by virtue of \eqref{eq:2c}, at
least locally, to give a map $g:\submfd\to G$ with $g^{-1}\d
g=\alpha$ and thus a map $j=gH:\submfd\to G/H$.  Since the system is
gauge-invariant, (replace $g$ by $gh$ for any $h:\submfd\to H$),
it is the map $j$ that carries the geometry. 

With $N=G/K$, we view $j:\submfd\to G/H\subset J(N)$ as a map into
twistor space and set $\phi=\pi\circ j:\submfd\to N$.  We may
therefore view $j$ as an orthogonal almost complex structure on
$\phi^{-1}TN$.  Any local frame $g:\submfd\to G$ with $j=gH$ and
$\alpha=g^{-1}\d g$ gives an isomorphism
$\phi^{-1}TN^\C\cong\submfd\times\p$ under which $j$ corresponds to
$\tau_{|\p}$; $\d\phi$ with $\alpha_1+\alpha_{-1}$ and
$\d+\alpha_0+\alpha_2$ with $\nb$, the Levi-Civita connection of $N$,
pulled back by $\phi$.  Let $D$ be the connection on $\phi^{-1}TN$
corresponding in this way to $\d+\alpha_0$.  From \eqref{eq:3}, we
have that $(\d+\alpha_0)\tau_{|\p}=0$ giving $Dj=0$ and also that
$\ad\alpha_2$ anti-commutes with $\tau_{|\p}$ whence $\nb-D$
anti-commutes with $j$.  It follows at once that $\alpha_2$
corresponds to $\half j(\nb j)$ so that $D=\nb-\half j(\nb j)$.  We
can now read off the significance of the equations \eqref{eq:2}:

First \eqref{eq:2a} is exactly the assertion that $\phi$ is
holomorphic with respect to $j$: $\d\phi\circ
J^\submfd=j\circ\d\phi$.  This means that $j$ is a \emph{twistor
lift} of $\phi$.  As a consequence, $\phi$ is a (branched) conformal
immersion. 

Equation \eqref{eq:2b} amounts to the demand that $\d^D
j(\nb j)^{1,0}=0$.  We then have the complex conjugate equation $\d^D
j(\nb j)^{0,1}=0$ and, together, these are equivalent to
\begin{subequations}
\label{eq:8}
\begin{align}
\label{eq:9}
0&=d^Dj(\nb j)\\
\label{eq:10}
0&=d^D*j(\nb j). 
\end{align}
\end{subequations}
The first of these is a consequence of \eqref{eq:2a} and
\eqref{eq:7}.  For this, write 
\begin{equation*}
\so(\phi^{-1}TN)=\so_+(\phi^{-1}TN)\oplus\so_-(\phi^{-1}TN)
\end{equation*}
for the $D$-parallel, symmetric decomposition of $\so(\phi^{-1}TN)$
into $j$-commuting and $j$-anti-commuting parts and let
$\pi_-:\so(\phi^{-1}TN)\to\so_-(\phi^{-1}TN)$ be the (orthogonal)
projection along $\so_+(\phi^{-1}TN)$.  Then $\nb j$ takes values in
$\so_-(\phi^{-1}TN)$ and, since
$[\so_-(\phi^{-1}TN),\so_-(\phi^{-1}TN)]\subset\so_+ (\phi^{-1}TN)$,
$D=\pi_-\circ\nb$ on $\so_- (\phi^{-1}TN)$.  We now compute:
\begin{equation*}
0=d^D j(\nb j)=j d^D\nb j=j\pi_-d^{\nb} \nb
j=j\pi_-[R^{\nb},j]=j[R^{\nb},j]. 
\end{equation*}
Thus \eqref{eq:9} is equivalent to $R^{\nb}$ commuting with $j$ which
is an immediate consequence of Theorem~\ref{th:1} along with the
observation that $\d\phi(T\submfd)$ is $j$-stable in view of
\eqref{eq:2a}. 

On the other hand, equation \eqref{eq:10} is the vertical part of a harmonic map
equation for $j$:
\begin{equation*}
0=* d^D * (j\nb j)=j\pi_-(* d^{\nb} * \nb j)=-j\pi_- \nb^*\nb j. 
\end{equation*}
This is exactly the condition that $j$ be a harmonic section of
$J(\phi^{-1}TN)$ in the sense of C.M.~Wood \cite[see
Theorem~4.2(c)]{Woo03}.  We say that such a twistor lift is
\emph{vertically harmonic}. 

To summarise:
\begin{thm}
\label{th:3}
Let $j:\submfd\to G/H \subset J(G/K)$ be a map of a Riemann surface
into the twistor space of a Riemannian symmetric space $G/K$ which
factors through a locally $4$-symmetric space as in
section~\ref{sec:4-symmetric-spaces}.  Let $\phi=\pi\circ j$. 

Then $j$ admits local frames $g$ with $\alpha=g^{-1}\d g$ solving
\eqref{eq:2} if and only if 
\begin{enumerate}
\item $j$ is a twistor lift: $j\circ\d\phi=\d\phi\circ J^\submfd$;
\item $j$ is vertically harmonic: $[\nb^*\nb j,j]=0$. 
\end{enumerate}
Moreover, in this case, $[R^{\nb},j]=0$. 
\end{thm}

It is interesting that our integrable system is solely concerned with
the geometry of $j$ \emph{qua} map into twistor space.  The only role
played by the locally $4$-symmetric space $G/H$ is to provide a (possibly
empty) algebraic constraint on $j$ and a curvature identity. 

\section{$4$ dimensions and holomorphic mean curvature vector}

We have seen that our theory is one about twistor lifts of conformal
immersions in a symmetric space.  In general, there are many such
lifts but, in favourable circumstances, there are distinguished lifts
and then our theory is one of the conformal immersions themselves. 

In particular, suppose that $\phi:\submfd\to N$ is a conformal
immersion into an oriented $4$-manifold.  The twistor space of $N$
has two components $J_{\pm}(N)$, each a $S^2$-bundle over $N$, and
there are unique twistor lifts $j_\pm:\submfd\to J_\pm(N)$ given by
choosing one of the two orthogonal almost complex structures on the
normal bundle of $\phi$.  The vertical harmonicity of these twistor
lifts has been studied by Hasegawa \cite{Has07}. 

In this situation, we have a splitting $\phi^{-1}TN=T\submfd\oplus
N\submfd$ and a corresponding decomposition
\begin{equation*}
\nb=
\begin{pmatrix}
\ns&-\II^t\\
\II&\np
\end{pmatrix}
\end{equation*}
where $\ns,\np$ are, respectively, the Levi-Civita and
normal connections of $\submfd$ and
$\II\in\Omega^1_\submfd\otimes\Hom(T\submfd,N\submfd)$
is the second fundamental form of the immersion.  

To explicate the vertical harmonicity of $j=j_\pm$, write
$\II=\II_++\II_-$ where $\II_+$ commutes with $j$ and $\II_-$
anti-commutes.  Both $j_{|T\submfd}$ and $j_{|N\submfd}$ are
rotations through $\pi/2$ and so $\ns j=\np j=0$.  We
deduce that
\begin{align*}
\half j\nb j&=\half j[\II-\II^t,j]=\II_--\II^t_-\\
D&=\ns+\np+\II_+-\II^t_+. 
\end{align*}
Moreover, $[(\II-\II^t)_+\wedge *(\II-\II^t)_-]$ must vanish as it takes values in
$\so(T\submfd)\oplus\so(N\submfd)$ and anti-commutes
with $j$ and we therefore conclude that $j$ is vertically harmonic if
and only if 
\begin{equation}
\label{eq:6}
\d^{\ns,\np}*\II_-=0. 
\end{equation}

In case $j$ factors through a locally $4$-symmetric space,
\eqref{eq:6} admits a simple interpretation thanks to \eqref{eq:7}
and the Codazzi equation.  For this, use $j_{|N\submfd}$ to view
$N\submfd$ as a complex line bundle and then equip that line bundle
with the Koszul--Malgrange holomorphic structure whose
$\bar\partial$-operator is $(\np)^{0,1}$.  We now have:
\begin{thm}\label{th:2}
Let $\phi:\submfd\to G/K$ be a conformal immersion of a Riemann
surface into a $4$-dimensional Riemannian symmetric space $G/K$ and
let $j:\submfd\to G/H\subset J(G/K)$ be a twistor lift of $\phi$
factoring through a locally $4$-symmetric space $G/H$ as in
Section~\ref{sec:4-symmetric-spaces}.   

Then $j$ is vertically harmonic if and only if the mean curvature
vector $\mcv=\half\trace\II$ is a holomorphic section of
$(N\submfd,j_{|N\submfd})$.
\end{thm}
\begin{proof}
Let $\nabla\II$ denote the covariant derivative of $\II$ with respect
to the connection on $T^*\submfd\otimes\Hom(T\submfd,N\submfd)$
induced by $\ns,\np$.  For $X\in T\submfd$, we have
\begin{equation}
\label{eq:11}
(*\d^{\ns,\np}*\II)X=(R^{\nb}(e_i,X)e_i)^\perp+2\np_X\mcv.
\end{equation}
Indeed, with $e_1, e_2=je_1$ a local frame of $T\submfd$,
\begin{align*}
(*\d^{\ns,\np}*\II)X&=(\trace\nabla\II)X
=(\nabla_{e_i}\II)_{e_i}X\\
&=(\nabla_{e_i}\II)_X e_i\\
&=(\d^{\ns,\np}\II)_{e_i,X}(e_i)+(\nabla_X\II)_{e_i}e_i\\
&=(\d^{\ns,\np}\II)_{e_i,X}(e_i)+\trace\nabla_X\II\\
&=(R^{\nb}(e_i,X)e_i)^\perp+2\np_X\mcv,
\end{align*}
where we have used, in order, the symmetry of $\II$; that
$\ns$ is torsion-free; the Codazzi equation and that $\trace$ is
$\ns$-parallel. 

Moreover, observe that the endomorphism $X\mapsto
(R^{\nb}(e_i,X)e_i)^\perp$ commutes with $j$ by virtue of
\eqref{eq:7}:
\begin{equation*}
j(R^{\nb}(e_i,X)e_i)^\perp=(R^{\nb}(je_i,jX)je_i)^\perp=(R^{\nb}(e_i,jX)e_i)^\perp
\end{equation*}
so that, applying the $\ns,\np$-parallel projection $\pi_-$ to
\eqref{eq:11}, we have:
\begin{equation*}
*\d^{\ns,\np}*\II_-=\pi_-(*\d^{\ns,\np}*\II)=2\pi_-(\np\mcv).
\end{equation*}
We conclude that $j$ is vertically harmonic if and only if $\np\mcv$
commutes with $j$ or, equivalently, $\mcv$ is a holomorphic section
of $(N\submfd,j_{|N\submfd})$.
\end{proof}
\subsection{Examples}

Let us enumerate the $4$-symmetric spaces that fibre over a
simply-connected $4$-dimensional Riemannian symmetric space $N$.  There
are three cases:

\subsubsection{ $N$ has constant sectional curvatures} Here both
$J_+(N)$ and $J_-(N)$ are themselves globally $4$-symmetric so there
is no algebraic constraint on $j$ to take into account.  We therefore
recover a theorem of Hasegawa:
\begin{thm}[\cite{Has07}]
A conformal immersion $\phi:\submfd\to N$ of a Riemann surface into a
simply connected $4$-manifold of constant sectional curvature 
has vertically harmonic twistor lift $j_\pm$ if and only if its mean
curvature vector is holomorphic in $(N\submfd,(j_\pm)_{|N\submfd})$. 
\end{thm}
It is interesting that these surfaces already solve an integrable
system in conformal geometry: they are constrained Willmore surfaces
(see \cite{BurCal} and also Bohle \cite{Boh06} for the case
$N=\mathbb{R}^4$).

\subsubsection{ $N$ has constant holomorphic sectional curvatures}
For any K\"ahler $4$-manifold $N$, fix the orientation so that the
ambient complex structure $J^N$ is a section of $J_+(N)$.   If $N$
has constant holomorphic sectional curvatures, $J_-(N)=\set{j\in
J(N)\colon [j,J^N]=0}$ is again $4$-symmetric (for example, if $N=\C
P^2$, $J_-(N)$ is the full flag manifold of $\mathrm{SU}(3)$).  Again
there is no algebraic constraint to take into account and we conclude:
\begin{thm}
A conformal immersion $\phi:\submfd\to N$ of a Riemann surface into a
simply connected $4$-manifold of constant holomorphic sectional curvature 
has vertically harmonic twistor lift $j_-$ if and only if its mean
curvature vector is holomorphic in $(N\submfd,(j_-)_{|N\submfd})$. 
\end{thm}

\subsubsection{$N$ is Hermitian symmetric}
Again we fix orientations so that the ambient complex structure $J^N$
is a section of $J_+(N)$ and now contemplate the subbundle
$Z=\set{j\in J_+(N): \{j,J^N\}=0}$ of almost complex
structures that anti-commute with $J^N$.  This is a circle bundle
over $N$ (it is the unit circle bundle in the canonical bundle of
$N$) and is $4$-symmetric.

In this case, we do have an algebraic constraint of the twistor lift
$j_+$.  Indeed, let $\omega^N$ be the
K\"ahler form of $N$.  We then have:
\begin{lemma}\label{th:4}
$j_+$ takes values in $Z$ if and only if $\phi$ is Lagrangian:
$\phi^*\omega^N=0$.
\end{lemma}
\begin{proof}
First, if $\phi$ is Lagrangian, $J^N$ restricts to isometries
$T\submfd\to N\submfd$ and $N\submfd\to T\submfd$ whence $J^N j_+=\pm
j_+ J^N$ and now, since $j_+$ takes values in $J_+(N)$, we must have $J^N j_+=-
j_+ J^N$.  Thus $j_+$ is $Z$-valued.

Conversely, if $\{j,J^N\}=0$ then one readily computes that, for
$X\in T\submfd$, 
\begin{equation*}
\phi^*\omega^N(X,J^\submfd
X)=-\phi^*\omega^N(X,J^\submfd X)
\end{equation*}
so that $\phi$ is Lagrangian.
\end{proof}
Now recall the \emph{Maslov form} of a Lagrangian immersion $\phi$
which is given\footnote{Here we have dropped a customary factor of
$1/\pi$.} by $\beta=\iota_\mcv \omega^N\in\Omega^1_\submfd$ .  We
have:
\begin{lemma}
Let $\phi:\submfd\to N$ be Lagrangian with twistor lift
$j=j_+:\submfd\to Z$.  Then
\begin{equation}
\label{eq:12}
\II_-=\beta J^N{}_{|T\submfd}.
\end{equation}
\end{lemma}
\begin{proof}
First note that since $N$ is K\"ahler, $\nb J^N=0$ whence
\begin{align}
\label{eq:13}
\II\circ J^N&=J^N\circ\II^t\\
\nabla J^N&=0\label{eq:14}
\end{align}
where, as before, $\nabla$ is the connection on $\so(TN)$ induced by
$\ns$ and $\np$.  Taking the $j$-anti-commuting part of \eqref{eq:13}
yields
\begin{equation}
\label{eq:15}
\II_-\circ J^N=J^N\circ(\II)_-^t.
\end{equation}
By definition $\II_-$ takes values in the subbundle of
$\Hom(T\submfd,N\submfd)$ consisting of homomorphisms that
anti-commute with $j$.  In the present situation, this subbundle is
spanned by $J^N{}_{|T\submfd}$ and $jJ^N{}_{|T\submfd}$ and we easily
see from \eqref{eq:15} that $\II$ points along $J^N{}_{|T\submfd}$
and so is of the form $\beta J^N{}_{|T\submfd}$ for some
$\beta\in\Omega^1_{\submfd}$.

It remains to identify $\beta$ with the Maslov form.  For this let
$X,e\in T_p\submfd$ with $e$ of unit length.  Then, using
$\II_-=\half(\II+j\II j)$, the skew-symmetry of $j$ and $J^N$ as well
as the fact that these anti-commute, we have
\begin{equation*}
\beta_X=(\II_-{}_X,J^Ne)=\half\bigl((\II_X,J^N e)+(\II_X je,J^N je)\bigr).
\end{equation*}
On the other hand, for any $Y\in T_p\submfd$, the symmetry of $\II$
along with \eqref{eq:13} yields
\begin{equation*}
(\II_X Y,J^N Y)=(\II_Y X,J^N Y)=(X,\II^t_Y J^N Y)=(X,J^N\II_Y Y)
\end{equation*}
so that we conclude that
\begin{equation*}
\beta_X=(X,J^N\mcv)
\end{equation*}
and we are done.
\end{proof}
With this in hand, the vertical harmonicity of $j_+$ is easy to
understand: in view of \eqref{eq:14}, we have
\begin{equation*}
\d^{\ns,\np}*\II_- =(\d *\beta)J^N{}_{|T\submfd}
\end{equation*}
so that $j_+$ is vertically harmonic if and only if the Maslov form
is co-closed which, in turn, is precisely the condition that $\phi$
be Hamiltonian stationary \cite{Oh93}.  To summarise:
\begin{thm}
Let $\phi:\submfd\to N$ be a conformal immersion of a Riemann surface
into a simply-connected $4$-dimensional Hermitian symmetric space
with twistor lift $j_+:\submfd\to J_+(N)$.

Then $j_+$ takes values in $Z=\set{j\in J_+(N): \{j,J^N\}=0}$ if and
only if $\phi$ is Lagrangian and then is vertically harmonic if and
only if $\phi$ is Hamiltonian stationary.
\end{thm}
This provides a conceptual explanation for the integrable system
appearing in the work of H\'elein--Romon
\cite{HelRom02,HelRom02A,HelRom05}.

\section{Prospects}

We have seen that the integrable system \eqref{eq:2} gives a
satisfying geometric theory for conformal immersions of Riemann
surfaces in $4$-dimensional symmetric spaces.  It is natural to ask
whether there are similar theories in case $\submfd$ has either
higher codimension in a symmetric space or, indeed, higher
dimension.  We briefly examine these possibilities.

\subsection{Higher codimension}
\label{sec:higher-co-dimension}

When $\dim N>4$, there are, in general, no canonically defined
twistor lifts.  However, there are still interesting examples.  Here
is one such: take $N=\mathbb{R}^8$, identified with the octonions and
contemplate the unit sphere $S^6\subset\mathrm{Im}\Oct$.  Left
multiplication by $q\in S^6$ is an orthogonal complex structure on
$\Oct$ so that we can view the trivial $S^6$-bundle over $N$ as a
submanifold of $J(N)$ which is once more a $4$-symmetric space.
Moreover, if $\phi:\submfd\to N$ is a conformal immersion of a
Riemann surface, we get a map $\submfd\to S^6$ by $p\mapsto
q_1\bar{q_2}$, for $q_1,q_2$ an oriented orthonormal basis of
$\d\phi(T_p\submfd)$, and thus a map $j:\submfd\to J(N)$.  It is easy
to see that $j$ is a twistor lift of $\phi$ and our theory applies to
show that the $\phi$ with vertically harmonic $j$ (the
$\rho$-harmonic surfaces of \cite{Khe05}) constitute an
integrable system.  Thus we recover the results of Khemar
\cite{Khe05}.

However, the geometric interpretation of $\rho$-harmonic surfaces is
not as straightforward as in codimension $2$: both $\np j$ and
$[\II^t_+\wedge\II^{1,0}_-]$ contribute non-trivial terms to
\eqref{eq:8}.  We may return to this elsewhere.

\subsection{Higher dimension}
\label{sec:higher-dimension}

It is well-known that the integrable systems approach to harmonic
maps of Riemann surfaces extends essentially without adjustment to
treat pluriharmonic maps of K\"ahler manifolds: see, for example,
\cite{BurFerPed93,OhnVal90}.  The same is true in the present case.

For this, with $\g,\tau$ as in Section
\ref{sec:an-integrable-system}, let $\submfd$ be a complex manifold
and $\alpha\in\Omega^1_{\submfd}\otimes\g$.  In this context, the
integrable system \eqref{eq:2} can be reformulated as:
\begin{subequations}
\label{eq:16}
\begin{gather}
\alpha^{0,1}_{1}=0 \label{eq:17}\\
\bar{\partial}\alpha^{1,0}_2+[\alpha^{0,1}_0\wedge\alpha^{1,0}_2]=0
\label{eq:18}\\
\d\alpha+\tfrac{1}{2}[\alpha\wedge\alpha]=0.\label{eq:19}
\end{gather}
\end{subequations}
This again has a zero curvature representation so long as
$\h=\g_0\cap\g$ is compact: indeed, with $\alpha_\lambda$
defined as before, the flatness of $\d+\alpha_\lambda$ amounts to
\eqref{eq:16} along with $[\alpha_2^{1,0}\wedge\alpha_2^{1,0}]=0$.
However, we can argue as in \cite{OhnVal90} starting with
$(\partial\bar{\partial}+\bar{\partial}\partial)\alpha_2^{1,0}=0$ and
using \eqref{eq:16} to conclude that
\begin{equation*}
[[\alpha_2^{1,0}\wedge\alpha_2^{1,0}]\wedge\alpha_2^{0,1}]=0.
\end{equation*}
The vanishing of $[\alpha_2^{1,0}\wedge\alpha_2^{1,0}]$ now follows
by contracting this against $\alpha_2^{0,1}$ and using the
definiteness of the Killing form on $\h$.

We can now repeat the analysis of Section~\ref{sec:vert-harm-twist},
\emph{mutatis mutandis}, to conclude:
\begin{thm}
Let $j:\submfd\to G/H \subset J(G/K)$ be a map of a complex manifold
into the twistor space of a Riemannian symmetric space $G/K$ which
factors through a locally $4$-symmetric space as in
section~\ref{sec:4-symmetric-spaces}.  Let $\phi=\pi\circ j$. 

Then $j$ admits local frames $g$ with $\alpha=g^{-1}\d g$ solving
\eqref{eq:16} if and only if 
\begin{enumerate}
\item $j$ is a twistor lift: $j\circ\d\phi=\d\phi\circ J^\submfd$;
\item $j$ is vertically pluriharmonic: $[\bar{\partial}^{\nb}(\nb j)^{1,0},j]=0$. 
\end{enumerate}
Moreover, in this case, $[R^{\nb},j]^{1,1}=0$. 
\end{thm}

\providecommand{\bysame}{\leavevmode\hbox to3em{\hrulefill}\thinspace}
\providecommand{\MR}{\relax\ifhmode\unskip\space\fi MR }
\providecommand{\MRhref}[2]{%
  \href{http://www.ams.org/mathscinet-getitem?mr=#1}{#2}
}
\providecommand{\href}[2]{#2}

\end{document}